\documentclass[11pt, letterpaper, oneside]{amsart}
\usepackage{tikz}
\usepackage{hyperref}
\usepackage{amsrefs}
\usepackage{amsthm}
\usepackage{amssymb}
\usepackage{mathtools}
\usepackage{float}

\newtheorem{theorem}{Theorem}[section]

\newtheorem{lemma}[theorem]{Lemma}
\newtheorem{proposition}[theorem]{Proposition}
\newtheorem{conjecture}[theorem]{Conjecture}

\theoremstyle{definition}

\newtheorem{remark}[theorem]{Remark}
\numberwithin{equation}{section}

\newcommand{\qb}[2]{\left[{{#1} \atop {#2}} \right]_q}

\begin{document}

\title[A requested analytic proof of an identity]{A requested analytic proof of an identity of Dixit, Kumar, and Srivastava}
\author{Philip Cuthbertson}
\address{Department of Mathematical Sciences\\ Michigan Technological University\\ Houghton, MI 49931} \email{pecuthbe@mtu.edu}

\begin{abstract}
    Recently Dixit, Kumar, and Srivastava investigated what they called Rascoe and non-Rascoe partitions. These are defined to be the set of distinct partitions where the length of the partition is a part of the partition and is not a part respectively. In this note we provide a $q$-series theoretic proof of two identities regarding the generating function for unrestricted Rascoe and non-Rascoe partitions fulfilling a request of Dixit, Kumar, and Srivastava. We also prove a conjecture of Beck relating non-Rascoe partitions and the rank of a partition.
\end{abstract}

\maketitle

\section{Introduction}
    Following \cite{Dixit-Kumar-Srivastava}, we define $a(n)$ as the number of partitions of $n$, $\lambda$, into distinct parts such that the number of parts of $\lambda$ is a part of $\lambda$. Additionally, we also define $b(n)$ as the number of partitions of $n$ into distinct parts that do not contain the length as a part. In \cite{Dixit-Kumar-Srivastava} Dixit, Kumar, and Srivastava coined the names \textit{Rascoe partitions} and \textit{non-Rascoe partitions} for these respectively. They also investigated \textit{unrestricted Rascoe partitions} and \textit{unrestricted non-Rascoe partitions} which are the same definition except without the restriction on the parts being distinct. We define $c(n)$ and $e(n)$ to be the analogous counts of these two sets of partitions. They also proved the following theorem for the generating functions of $c(n)$ and $e(n)$ combinatorially and then asked for a purely $q$-series proof of the identity. In this note we fulfill this request using well known identities such as the $q$-Pascal identity, $q$-binomial theorem, $q$-Gauss identity, and a lesser known identity involving a terminating $_2\phi_1$ basic hypergeometric series.
\begin{theorem}[Dixit-Kumar-Srivastava]\label{mainTheorem}
    For $|q| < 1$, we have
    \[\sum_{n = 0}^\infty c(n)q^n = q + \sum_{n = 2}^\infty \sum_{m = 0}^{n} \qb{2n - m - 2}{n - m - 1} \frac{q^{mn + 2n - 1}}{(q; q)_m} = \frac{q}{(q^2; q)_\infty}.\]
    Additionally,
    \[\sum_{n = 0}^\infty e(n)q^n = 1 + \frac{q^2}{1 - q} + \sum_{n = 2}^\infty \sum_{m = 0}^{n} \qb{2n - m - 2}{n - m} \frac{q^{mn + n}}{(q; q)_m} = \frac{1 - q + q^2}{(q; q)_\infty}.\]
\end{theorem}

Immediate from the definitions and the rightmost generating functions we have that $c(n) + e(n) = p(n)$ for all $n$. As such, our $q$-series proofs will be split into showing the identity for $\sum c(n) q^n$ and then showing that $\sum (c(n) + e(n)) q^n$ is equal to the partition generating function.

In Section 2 we will recall some classical results and definitions for the notation used throughout, Section 3 will be dedicated to a proof of Theorem \ref{mainTheorem}, in Section 4 we give a combinatorial proof of a related conjecture of Beck \cite{Beck-OEIS}, and in Section 5 we provide a refined congruence conjecture and discuss future work.

\section{Background}

The sequence $\lambda = (\lambda_1, \lambda_2, \ldots, \lambda_m)$ is called an \textit{integer partition} of $n$ if $\lambda$ is a weakly decreasing sequence of positive integers that sum to $n$. We call $\ell(\lambda) := m$ the \textit{length} and $|\lambda| := n$ the \textit{size} of $\lambda$. Additionally we let $p(n)$ count the number of partitions of $n$. We will also make use of some standard notation that we will use and provide several identities that will be used in the algebraic proofs. The $q$-Pochhammer symbol is defined as
\begin{align*}
    (a; b)_n & := \prod_{k=0}^{n-1}(1-ab^k),\\
    (a; b)_\infty & := \prod_{k=0}^{\infty}(1-ab^k),\\
    (a_1, a_2, \dots, a_r; q)_n & := (a_1; q)_n (a_2; q)_n \cdots (a_r; q)_n.
\end{align*}
It will be useful to also have the $q$-binomial coefficient which is given by
\[\qb{a}{b} := \frac{(q; q)_a}{(q; q)_b(q; q)_{a-b}}.\]
We will also make use of the $q$-hypergeometric series which is defined by
\[{}_r\phi_s \left[ \begin{matrix} a_1, a_2, \dots, a_r \\ b_1, b_2, \dots, b_s \end{matrix} ; q, z \right] := \sum_{n=0}^\infty {\frac {(a_1, a_2, \ldots ,a_r; q)_n}{(b_1, b_2,\ldots, b_s, q;q)_n}}\left((-1)^n q^{n \choose 2}\right)^{1 + r - s} z^n,\]
for $r$ and $s$ non-negative integers.

\begin{theorem}[$q$-Pascal \cite{Gasper-Rahman}*{p. 353, Eq. (I.45)}]\label{qPascal}
    For $a, b, q \in \mathbb C$,
    \[\qb{a}{b} = q^b \qb{a-1}{b} + \qb{a}{b - 1} = \qb{a-1}{b} + q^{a - b}\qb{a - 1}{b - 1}.\]
\end{theorem}

\begin{theorem}[$q$-Binomial Theorem \cite{Gasper-Rahman}*{p. 354, Eq. (II.3)}]\label{qBinomial}
Given $|q| < 1$ and $|z| < 1$
    \[{}_1\phi_0 \left[ \begin{matrix} a \\ {} \end{matrix} ; q, z \right] = \sum_{n=0}^{\infty} \frac{(a;q)_n}{(q;q)_n} z^n = \frac{(az; q)_\infty}{(z; q)_\infty}.\]
\end{theorem}

\begin{theorem}[$q$-Gauss sum \cite{Gasper-Rahman}*{p. 354, Eq. (II.8)}]\label{qGauss}
Given $|q| < 1$ and $|c/(ab)| < 1$
    \[{}_2\phi_1 \left[ \begin{matrix} a, b \\ c \end{matrix} ; q, c/(ab) \right] = \frac{(c/a; q)_\infty (c/b; q)_\infty}{(c; q)_\infty (c/(ab); q)_\infty}.\]
\end{theorem}

\begin{theorem}[\cite{Gasper-Rahman}*{p. 359, Eq. (III. 8)}]\label{finite2phi1}
    \[{}_2\phi_1 \left[\begin{matrix} q^{-n}, b \\ c \end{matrix} ; q, z \right] = \frac{(c/b; q)_n}{(c; q)_n} b^n {}_3\phi_1 \left[ \begin{matrix} q^{-n}, b, q/z \\ bq^{1 - n}/c \end{matrix} ; q, \frac{z}{c} \right].\]
\end{theorem}

\begin{proposition}[\cite{Gasper-Rahman}*{p. 351, Eq. (I.3)}]\label{pochhammerFactorization}
    \[(a; b)_n = (-1)^n a^n b^{\binom{n}{2}} (a^{-1}, b^{-1})_n.\]
\end{proposition}

\section{Proof}

We begin with the following two lemmas.

\begin{lemma}\label{lemma1}
    \[\sum_{N = 0}^\infty \frac{q^{2N+1}}{(q^{-N};q)_N} {}_2\phi_1\left[ \begin{matrix} q^{-N}, q^{N+1} \\ q \end{matrix} ; q, 1 \right] = \frac{q}{(q^2; q)_\infty}.\]
\end{lemma}

\begin{proof}
Consider the following slightly modified series
\[\sum_{N=0}^\infty \frac{q^{2N+1}}{(xq^{-N};q)_N} {}_2 \phi_1 \left[
\begin{matrix} q^{-N}, q^{N+1} \\ xq \end{matrix} ; q , x \right].\]
We apply the $q$-Gauss identity (Theorem \ref{qGauss}) with $a = q^{-N}, b = q^{N + 1},$ and $c = xq$ to make this
\[\sum_{N = 0}^\infty \frac{q^{2N + 1} (xq^{N + 1}; q)_\infty
(xq^{-N}; q)_\infty}{(xq^{-N}; q)_N (xq; q)_\infty (x; q)_\infty} .\]
Since 
\[\frac{(xq^{-N}; q)_\infty}{(xq^{-N}; q)_N (x; q)_\infty} = 1\]
this now simplifies to 
\[\sum_{N=0}^\infty \frac{q^{2N+1} (xq^{N+1};q)_\infty}{(xq;q)_\infty} = q\sum_{N=0}^\infty \frac{q^{2N}}{(xq;q)_N}.\]
We can take the limit $x \to 1$ avoiding $x = q^j$ for any $j$ and then apply the $q$-binomial theorem (Theorem \ref{qBinomial}) with $a = 0$ and $z = q^2$ to prove the lemma.
\end{proof}

\begin{lemma}\label{lemma2}
    \[1 + \sum_{N = 0}^\infty \frac{q^{(N + 1)(N + 2)}}{(q; q)_{N + 1}}{}_3\phi_1 \left[\begin{matrix} q^{-(N + 1)}, q^{N+1}, q \\ q \end{matrix} ; q, 1 \right] = \frac{1}{(q; q)_\infty}.\]
\end{lemma}

\begin{proof}
Let $0 \neq |x| < 1$ and consider the following slightly modified series
\[1 + \sum_{N = 0}^\infty \frac{q^{(N + 1)(N + 2)}}{(q; q)_{N + 1}}{}_3\phi_1 \left[\begin{matrix} q^{-(N + 1)}, q^{N+1}, q/x \\ q/x \end{matrix} ; q, 1/x \right].\]
Using \ref{finite2phi1} with $n = N + 1, b = q^{N + 1}, c = x,$ and $z = x$, we obtain
\[1 + \sum_{N = 0}^\infty \frac{q^{(N + 1)(N + 2)}}{(q; q)_{N + 1}}
\frac{(x; q)_{N + 1}}{(x/q^{N + 1}; q)_{N + 1}}\left(\frac{1}{q^{N + 1}}\right)^{N + 1} {}_2\phi_1 \left[ \begin{matrix} q^{-(N + 1)}, q^{N + 1} \\ x \end{matrix} ; q, x \right].\]
Now applying \ref{qGauss} with $a = q^{-(N + 1)}, b = q^{N + 1},$ and $c = x$, this transforms into
\[1 + \sum_{N = 0}^\infty \frac{q^{(N + 1)(N + 2)}}{(q; q)_{N + 1}}
\frac{(x; q)_{N + 1}}{(x/q^{N + 1}; q)_{N + 1}}\left(\frac{1}{q^{N + 1}}\right)^{N + 1} \frac{(x/q^{-(N + 1)}; q)_\infty (x/q^{N + 1}; q)_\infty}{(x; q)_\infty (x/(q^{-(N + 1)} \cdot q^{N + 1}); q)_\infty}.\]
Now since
\[\frac{(x; q)_{N + 1}(xq^{N + 1}; q)_\infty}{(x; q)_\infty} \cdot \frac{(x/q^{N + 1}; q)_\infty}{(x/q^{N + 1}; q)_{N + 1}(x; q)_\infty} = 1,\]
the series simplifies to
\[1 + \sum_{N = 0}^\infty \frac{q^{N + 1}}{(q; q)_{N + 1}} = \sum_{m = 0}^\infty \frac{q^m}{(q; q)_m} = \frac{1}{(q; q)_\infty}.\]
Where we reindexed the sum and then applied Theorem \ref{qBinomial} with $a = 0$ and $z = q$. Now taking $x \to 1$ in the original modified series proves the lemma.
\end{proof}

We can now prove Theorem \ref{mainTheorem}.

\begin{proof}[Proof of Theorem \ref{mainTheorem}]
We start with the generating function
\[q + \sum_{n = 2}^\infty \sum_{m = 0}^{n} \qb{2n-m-2}{n-m-1} \frac{q^{mn + 2n - 1}}{(q; q)_m}.\]
which, after setting $N = n - 1$ and $M = N - m$, equals 
\[\sum_{N = 0}^\infty \sum_{M = 0}^{N} \qb{N + M}{M} \frac{q^{(N - M)(N + 1) + 2N + 1}}{(q; q)_{N - M}}.\]
Since $(q; q)_{N + M} = (q; q)_{N}(q^{N + 1}; q)_{M}$ and $(q; q)_{N - M} = (q; q)_N/(q^{N}; q^{-1})_{M}$, this becomes
\[\sum_{N = 0}^\infty \frac{q^{N^2 + 3N + 1}}{(q; q)_N} \sum_{M = 0}^{N} \frac{q^{- M(N + 1)}(q^{N + 1}; q)_{M}(q^N; q^{-1})_{M}}{(q; q)_{M}}.\]
Now we use Proposition \ref{pochhammerFactorization} with $a = q^N, b = q^{-1},$ and $n = M$ and simultaneously multiply by $(q; q)_M/(q; q)_M$ to obtain
\[\sum_{N = 0}^\infty \frac{q^{N^2 + 3N + 1}}{(q; q)_N} \sum_{M = 0}^{N} \frac{(-1)^M (q^{-1})^M (q^{-1})^{\binom{M}{2}}(q^{N + 1}; q)_{M}(q^{-N}; q)_{M}(q; q)_{M}}{(q; q)_{M}(q; q)_{M}}.\]
Finally, using Theorem \ref{finite2phi1} with $n = N, b = q^{N + 1},$ and $c = z = q$ we obtain
\[\sum_{N=0}^\infty \frac{q^{N^2+3N+1} (q; q)_N}{(q; q)_N(q^{-N};q)_N}
q^{-N^2 - N} {}_2\phi_1 \left[ \begin{matrix} q^{-N}, q^{N+1} \\ q
\end{matrix} ; q, 1 \right].\]
After canceling the power of $q$ and applying Lemma \ref{lemma1} we get the first half of the theorem. We will now focus on the sum of the two generating functions given in Theorem \ref{mainTheorem}.
\begin{flalign*}
    1 &+ q + \frac{q^2}{1 - q} + \sum_{n = 2}^\infty \sum_{m = 0}^{n}\left(\qb{2n - m - 2}{n - m - 1} \frac{q^{mn + 2n - 1}}{(q; q)_m} + \qb{2n - m - 2}{n - m} \frac{q^{mn + n}}{(q; q)_m}\right) \\
    &= 1 + q + \frac{q^2}{1 - q} + \sum_{n = 2}^\infty \sum_{m = 0}^{n}\frac{q^{mn + n}}{(q; q)_m} \left(q^{n - 1}\qb{2n - m - 2}{n - m - 1} + \qb{2n - m - 2}{n - m}\right).
\end{flalign*}
Where we can use the second equality in Theorem \ref{qPascal} with $a = 2n - m - 1$ and $b = n - m$ to obtain
\[1 + q + \frac{q^2}{1 - q} + \sum_{n = 2}^\infty \sum_{m = 0}^{n}\frac{q^{mn + n}}{(q; q)_m} \qb{2n - m - 1}{n - m}\]
\[= 1 + \sum_{n = 1}^\infty \sum_{m = 0}^{n}\frac{q^{mn + n}}{(q; q)_m} \qb{2n - m - 1}{n - m}.\]
After setting $N = n - 1$ and $M = N + 1 - m$, equals 
\[1 + \sum_{N = 0}^\infty \sum_{M = 0}^{N + 1} \frac{q^{(N + 1 - M)(N + 1) + N + 1}}{(q; q)_{N + 1 - M}} \qb{N + M}{M}\]
\[= 1 + \sum_{N = 0}^\infty \frac{q^{(N + 1)(N + 2)}}{(q; q)_N} \sum_{M = 0}^{N + 1} \frac{q^{-M(N + 1)} (q;q)_{N + M}}{(q;q)_{N + 1 - M}(q;q)_{M}}.\]
Since $(q; q)_{N + M} = (q; q)_{N}(q^{N + 1}; q)_{M}$ and $(q; q)_{N + 1 - M} = (q; q)_{N + 1}/(q^{N + 1}; q^{-1})_{M}$, this becomes
\[1 + \sum_{N = 0}^\infty \frac{q^{(N + 1)(N + 2)}}{(q; q)_N (1 - q^{N + 1})} \sum_{M = 0}^{N + 1} \frac{q^{-M(N + 1)} (q^{N+1};q)_{M} (q^{N+1}; q^{-1})_M}{(q; q)_{M}}.\]
Again we use Proposition \ref{pochhammerFactorization} with $a = q^{N + 1}, b = q^{-1},$ and $n = M$ and simultaneously multiply by $(q; q)_M/(q; q)_M$ to obtain
\[1 + \sum_{N = 0}^\infty \frac{q^{(N + 1)(N + 2)}}{(q; q)_{N + 1}} \sum_{M = 0}^{N + 1} \frac{(q^{-(N + 1)}; q)_M (q^{N + 1};q)_{M} (-1)^M q^{-\binom{M}{2}} (q;q)_{M}}{(q;q)_{M}(q;q)_{M}}.\]
Recognizing the inner sum as a ${}_3\phi_1$ and applying Lemma \ref{lemma2} now proves the theorem.
\end{proof}

\section{Proof of Beck's conjecture/Remark 10}

In \cite{Beck-OEIS} Beck conjectured that $e(n)$, defined in the introduction, is equal to the sum of the number of part sizes in all partitions of $2n + 2$ with rank $n + 1$, where the \textit{rank} of a partition is defined to be the largest part minus the number of parts. We give a combinatorial proof of this fact without relying on generating functions. The following two lemmas are not new, but we include proofs of both for completeness.

\begin{lemma}
    The total number of partitions of $n$ that do not contain a part of size 1 is equal to $p(n) - p(n - 1)$.
\end{lemma}

\begin{proof}
    We create a bijection between the set of partitions of $n$ that do contain a part of size 1 and the set of partitions of $n - 1$. Let $\lambda$ be a partition of $n$ containing a part of size 1. We remove this part of size 1 to obtain a partition of $n - 1$. Additionally, take a partition of $n - 1$ and append a part of size 1 to obtain a partition of $n$ with at least one part of size 1. These two actions form a bijection and thus there are $p(n - 1)$ partitions of $n$ that contain a part of size $1$. The difference $p(n) - p(n - 1)$ is thus the number of partitions of $n$ that do not contain any parts of size 1.
\end{proof}

\begin{lemma}
    The total number of part sizes across all partitions of $n$ that do not contain parts of size 1 is equal to $p(n - 2)$.
\end{lemma}

\begin{proof}
    We will create a bijection between the set of pairs $(\lambda, k)$ where $\lambda$ is a partition that does not contain any parts of size 1 and $k \geq 2$ is a part size that appears in $\lambda$ and the set of partitions $\mu$ of $n - 2$. For a pair $(\lambda, k)$ on the left form $\mu$ by removing a part of size $k$ from $\lambda$ and simultaneously appending $k - 2 \geq 0$ parts of size 1. Notice that $|\mu| = (n - k) + (k - 2) = n - 2$ so $\mu$ is a partition of size $n - 2$. In the other direction, take some partition $\mu$ of $n - 2$ and let $m_1$ be the number of parts of size 1 in $\mu$. We remove all $m_1$ of these parts and append a part of size $k = m_1 + 2 \geq 2$ to form $\lambda$, a partition of $(n - 2 - m_1) + (m_1 + 2) = n$ that contains no parts of size 1 with $k$ as a part size. Both of these maps are injections and form a bijection between the two sets proving the lemma.
\end{proof}

\begin{theorem}
    Let $f(n)$ be the sum of the number of part sizes in all partitions of $2n + 2$ with rank $n + 1$. Then,
    \[e(n) = p(n) - p(n - 1) + p(n - 2) = f(n).\]
\end{theorem}

\begin{proof}
    Notice that the first equality is an immediate consequence of Theorem \ref{mainTheorem}. Now let $\lambda$ be a partition of $2n + 2$ with rank $n + 1$. By definition we have $\lambda_1 = n + 1 + \ell(\lambda)$. Let $\lambda^*$ be the partition obtained after removing the largest part, $\lambda_1$, and adding 1 to each of the $\ell(\lambda) - 1$ remaining parts. Notice $\lambda^*$ is a partition of $(2n + 2) - (n + 1 + \ell(\lambda)) + (\ell(\lambda) - 1) = n$ with no parts of size 1. Also, the number of part sizes of $\lambda^*$ is exactly one less than that of $\lambda$ since $2 \lambda_1 > 2n + 2$ and thus $\lambda_1$ could not have been repeated in $\lambda$. If we let $D(\lambda)$ be the number of part sizes of $\lambda$, then $c(n)$ is given by
    \[f(n) = \sum_{\substack{\lambda \vdash 2n + 2 \\ rk(\lambda) = n + 1}} D(\lambda) = \sum_{\substack{\lambda^* \vdash n \\ m_1(\lambda) = 0}} (1 + D(\lambda^*)) = \sum_{\substack{\lambda^* \vdash n \\ m_1(\lambda) = 0}} 1 + \sum_{\substack{\lambda^* \vdash n \\ m_1(\lambda) = 0}} D(\lambda^*).\]
    Where $\mu \vdash n$ means that $\mu$ is a partition of $n$ and $m_1(\mu)$ is the number of parts of size 1 in $\mu$. Looking at the extreme right hand side, the previous two lemmas imply that the first sum is $p(n) - p(n - 1)$ and the second is $p(n - 2)$ proving the theorem.
\end{proof}

\section{Congruence Conjecture/Future Work}

In \cite{Dixit-Kumar-Srivastava}, Dixit, Kumar, and Srivastava conjectured a potential congruence modulo four for the number of non-Rascoe partitions, that is partitions into distinct parts where the length is not a part of the partition, of size $n$.

\begin{conjecture}[\cite{Dixit-Kumar-Srivastava} Conjecture 1]\label{Conjecture DKS}
    Let $b(n)$ denote the number of non-Rascoe partitions of an integer $n$. For $k \geq 1$ and $k$ not a multiple of 29, the following congruence holds:
    \[b(29k + 21) \equiv 0 \pmod 4.\]
\end{conjecture}

We further conjecture the following.

\begin{conjecture}\label{Conjecture C}
    Let $b(n)$ denote the number of non-Rascoe partitions of an integer $n$. For $j \geq 1$, the following congruence holds:
    \[b(29^2j + 21) + b(j) \equiv 0 \pmod 4.\]
\end{conjecture}

\begin{remark}
    In \cite{Dixit-Kumar-Srivastava}, Conjecture \ref{Conjecture DKS} was checked up to $n = 10^5$. We have checked both Conjecture \ref{Conjecture DKS} and Conjecture \ref{Conjecture C} up to $n = 10^6$.
\end{remark}


\begin{remark}
    While currently being unable to prove these conjectures in full, we note as to why $b(29k + 21)$ may act differently for $k$ a multiple of 29. Theorem 1.3 in \cite{Dixit-Kumar-Srivastava} gives that $b(n)$ is odd if and only if $n = m(5m + 1)/2$ for some integer $m$. Notice that if $n = m(5m + 1)/2$ then
    \[40n + 1 = (10m + 1)^2,\]
    so $40n + 1$ must be a perfect square. Notice then that
    \[40(29k + 21) + 1 = 29(40k + 29)\]
    can only be a perfect square if $k$ is also a multiple of 29. This does imply the modulo 2 version of Conjecture \ref{Conjecture C} since for $n = 29k + 21$ and $k = 29j$ then
    \begin{flalign*}
        40n + 1 &= 40(29k + 21) + 1 \\
        &= 29(40k + 29) \\
        &= 29(40 \cdot 29j + 29) \\
        &= 29^2(40j + 1)
    \end{flalign*}
    Thus $40n + 1$ is a perfect square precisely when $40j + 1$ is thus
    \[b(29^2j + 21) = b(n) \equiv b(j) \pmod 2.\]
    The same type of behavior exists for other progressions. For example, similar computations show that $b(11^2j + 3) \equiv b(j)$. This, of course, doesn't imply the entire modulo 4 behavior.
\end{remark}

In \cite{Dixit-Kumar-Srivastava}, Dixit, Kumar, and Srivastava posed additional problems none of which we are currently able to resolve but are of interest to the present author as well.


\end{document}